\documentclass{article}
\usepackage{mathrsfs, amsmath, amsthm}
\usepackage{graphicx, color}
\usepackage{caption, subcaption}
\usepackage{array}
\usepackage{cases}
\makeatletter
\def\th@plain{%
  \upshape 
}
\makeatother

\makeatletter
\renewenvironment{proof}[1][\proofname]{\par
  \pushQED{\qed}%
  \normalfont \topsep6\p@\@plus6\p@\relax
  \trivlist
  \item[\hskip\labelsep
        \bfseries
    #1\@addpunct{.}]\ignorespaces
}{%
  \popQED\endtrivlist\@endpefalse
}
\makeatother

\newtheorem{theorem}{Theorem}
\numberwithin{theorem}{section}
\newtheorem{lemma}{Lemma}

\newtheorem{conjecture}{Conjecture}
\newtheorem*{conjecture*}{Conjecture}

\newtheorem{claim}{Claim}
\newtheorem{fact}{Fact}

\theoremstyle{definition}

\usepackage[left=25mm,top=25mm,bottom=25mm,right=25mm]{geometry}
\usepackage[T1]{fontenc}
\usepackage[varg]{txfonts}

\usepackage{enumitem}
\usepackage[square, numbers, sort&compress]{natbib}

\ifx\pdfoutput\undefined
 \usepackage[dvipdfm,%
  bookmarks=true,%
  bookmarksnumbered=true, 
  bookmarksopen=true, 
  plainpages=false,%
  pdfpagelabels,%
  colorlinks=true, 
  linkcolor=blue, 
  citecolor=blue,%
  hyperindex=true,
  urlcolor=black]{hyperref}
\else
 \usepackage[pdftex,%
  bookmarks=true,%
  bookmarksnumbered=true, 
  bookmarksopen=true, 
  plainpages=false,%
  pdfpagelabels,%
  colorlinks=true, 
  linkcolor=blue, 
  citecolor=blue,%
  anchorcolor=green,
  urlcolor= blue,
  breaklinks=true,
  hyperindex=true]{hyperref}
\fi

\newcounter{Hcase}
\newcounter{Hclaim}

\newcommand{\resetcounter}{\stepcounter{Hcase}\setcounter{case}{0}\stepcounter{Hclaim}\setcounter{claim}{0}}

\newcommand{\etal}{et~al.\ }

\def\int(#1){\mathrm{int}(#1)}
\def\ext(#1){\mathrm{ext}(#1)}
\def\Int(#1){\mathrm{Int}(#1)}
\def\Ext(#1){\mathrm{Ext}(#1)}
\def\ad(#1){\mathrm{ad}(#1)}
\def\mad(#1){\mathrm{mad}(#1)}
\def\la(#1){\mathrm{la}(#1)}
\newcommand{\Lfloor}{\left\lfloor}
\newcommand{\Rfloor}{\right\rfloor}

\newcommand{\mul}{\mathrm{mul}}
\begin{document}
\title{Acyclic chromatic index of triangle-free $1$-planar graphs}
\author{Jijuan Chen\qquad Tao Wang\footnote{{\tt Corresponding
author: wangtao@henu.edu.cn}}  \qquad Huiqin Zhang\\
{\small Institute of Applied Mathematics}\\
{\small Henan University, Kaifeng, 475004, P. R. China}}
\date{August 20, 2016}
\maketitle
\begin{abstract}
An acyclic edge coloring of a graph $G$ is a proper edge coloring such that every cycle is colored with at least three colors. The acyclic chromatic index $\chiup_{a}'(G)$ of a graph $G$ is the least number of colors in an acyclic edge coloring of $G$. It was conjectured that $\chiup'_{a}(G)\leq \Delta(G) + 2$ for any simple graph $G$ with maximum degree $\Delta(G)$. A graph is {\em $1$-planar} if it can be drawn on the plane such that every edge is crossed by at most one other edge. In this paper, we prove that every triangle-free $1$-planar graph $G$  has an acyclic edge coloring with $\Delta(G) + 16$ colors. 


\end{abstract}

\section{Introduction}
All graphs considered in this paper are simple, undirected and finite. For a plane graph $G$, we use $F(G)$ to denote the face set of $G$. In a plane graph $G$, the degree of a face $f$, denoted $\deg(f)$, is the length of the boundary walk. A $k$-vertex, $k^{-}$-vertex and $k^{+}$-vertex is a vertex with degree $k$, at most $k$ and at least $k$, respectively. Analogously, we can define a $k$-face, $k^{-}$-face and $k^{+}$-face. 

An acyclic edge coloring of a graph $G$ is a proper edge coloring such that every cycle is colored with at least three colors. In other words, the union of any two color classes induces a subgraph such that every component is a path. The acyclic chromatic index $\chiup_{a}'(G)$ of a graph $G$ is the least number of colors in an acyclic edge coloring of $G$. 

By Vizing's theorem, the acyclic chromatic index of $G$ has a trivial lower bound $\Delta(G)$. Fiam\v{c}{\'i}k \cite{MR526851} stated the following conjecture in 1978, which is well known as Acyclic Edge Coloring Conjecture, and Alon \etal \cite{MR1837021} restated it in 2001.

\begin{conjecture}\label{graph}
For any graph $G$, $\chiup'_{a}(G)\leq \Delta(G) + 2$.
\end{conjecture}

Alon \etal \cite{MR1109695} proved that $\chiup'_{a}(G) \leq 64 \Delta(G)$ by using probabilistic method. Molloy and Reed \cite{MR1715600} improved it to $\chiup'_{a}(G) \leq 16 \Delta(G)$. Ndreca \etal \cite{MR2864444} improved the upper bound to $\lceil 9.62(\Delta(G)-1)\rceil$. Recently, Esperet and Parreau \cite{MR3037985} further improved it to $4\Delta(G) - 4$ by using the so-called entropy compression method. To my knowledge, the best known general bound is $\lceil 3.74(\Delta(G)-1)\rceil$ due to Giotis \etal \cite{MR3448625}. Alon \etal \cite{MR1837021} proved that there is a constant $c$ such that $\chiup'_{a}(G)\leq \Delta(G) + 2$ for a graph $G$ whenever the girth is at least $c\Delta \log\Delta$.

Regarding general planar graph $G$, Fiedorowicz \etal \cite{MR2458434} proved that $\chiup'_{a}(G)\leq 2\Delta(G)+29$, and Hou \etal \cite{MR2491777} proved that $\chiup'_{a}(G)\leq \max\{2\Delta(G)-2, \Delta(G)+22\}$. Recently, Basavaraju \etal \cite{MR2817509} showed that $\chiup_{a}'(G) \leq \Delta(G) + 12$, and Guan \etal \cite{MR3031510} improved it to $\chiup_{a}'(G) \leq \Delta(G) + 10$, and Wang \etal \cite{MR2994403} further improved it to $\chiup_{a}'(G) \leq \Delta(G) + 7$. The current best upper bound is $\Delta(G) + 6$ by Wang and Zhang \cite{MR3453934}. 

A graph is {$1$-planar} if it can be drawn on the plane such that every edge crosses at most one other edge. Obviously, the class of $1$-planar graphs is a larger class than the one of planar graphs. The concept of $1$-planar graph was introduced by Ringel \cite{MR0187232} in 1965, while he simultaneously colored the vertices and faces of a plane graph such that any pair of adjacent/incident elements receive distinct colors.

The Acyclic Edge Coloring Conjecture was verified for the triangle-free planar graphs, see \cite{MR3065111, MR3166127}. Recently, Song and Miao \cite{MR3397083} firstly studied the acyclic chromatic index of triangle-free $1$-planar graphs, and gave the following result. 
\begin{theorem}[Song and Miao \cite{MR3397083}]\label{Song}
If $G$ is a triangle-free $1$-planar graph, then $\chiup_{a}'(G) \leq \Delta(G) + 22$. 
\end{theorem}

Note that the upper bound $\Delta(G) + 22$ is far from the conjectured bound $\Delta(G) + 2$. In this paper, we improve the upper bound to $\Delta(G) + 16$, and we believe it can be further improved.  
\begin{theorem}\label{M}
If $G$ is a triangle-free $1$-planar graph, then $\chiup_{a}'(G) \leq \Delta(G) + 16$. 
\end{theorem}

\section{Preliminary and structural results}
Let $\mathbb{S}$ be a multiset and $x$ be an element in $\mathbb{S}$. The {\em multiplicity} $\mul_{\mathbb{S}}(x)$ is the number of times $x$ appears in $\mathbb{S}$. Let $\mathbb{S}$ and $\mathbb{T}$ be two multisets. The union of $\mathbb{S}$ and $\mathbb{T}$, denoted by $\mathbb{S} \uplus \mathbb{T}$, is a multiset with $\mul_{\mathbb{S} \uplus \mathbb{T}}(x) = \mul_{\mathbb{S}}(x) + \mul_{\mathbb{T}}(x)$. A graph $G$ with maximum degree at most $\kappa$ is {\em $\kappa$-deletion-minimal} if $\chiup_{a}'(G) > \kappa$ and $\chiup_{a}'(H) \leq \kappa$ for every proper subgraph $H$ of $G$.  

A {\em partial acyclic edge coloring} of $G$ is an acyclic edge coloring of any subgraph of $G$. Let $\phi$ be a partial acyclic edge coloring of $G$. Let $\mathcal{U}_{\phi}(v)$ denote the set of colors used on the edges incident with $v$. Let $A_{\phi}(v) = \{1, 2, \dots, \kappa\} \setminus \mathcal{U}_{\phi}(v)$ and $A_{\phi}(uv) = \{1, 2, \dots, \kappa\} \setminus (\mathcal{U}_{\phi}(u) \cup \mathcal{U}_{\phi}(v))$. Let $\Upsilon_{\phi}(u, v) = \mathcal{U}_{\phi}(v) \setminus \{\phi(uv)\}$ and $W_{\phi}(u, v) = \{\,u_{i} \mid uu_{i} \in E(G) \mbox{ and } \phi(uu_{i}) \in \Upsilon(u, v)\,\}$. Notice that $W_{\phi}(u, v)$ may be not the same as $W_{\phi}(v, u)$. For simplicity, we will omit the subscripts if no confusion can arise. 

An $(\alpha, \beta)$-maximal bichromatic path with respect to $\phi$ is a maximal path whose edges are colored by $\alpha$ and $\beta$ alternately. An $(\alpha, \beta, u, v)$-critical path with respect to $\phi$ is an $(\alpha, \beta)$-maximal bichromatic path which starts at $u$ with $\alpha$ and ends at $v$ with $\alpha$. An $(\alpha, \beta, u, v)$-alternating path with respect to $\phi$ is an $(\alpha, \beta)$-bichromatic path starting at $u$ with $\alpha$ and ending at $v$ with $\beta$.

A color $\alpha$ is {\em available} for an edge $e$ in $G$ with respect to a partial edge coloring of
$G$ if none of the adjacent edges of $e$ is colored with $\alpha$.  An available color $\alpha$ is {\em valid} for an edge $e$ if assigning the color $\alpha$ to $e$ does not result in any bichromatic cycle in $G$.

\begin{fact}[Basavaraju \etal \cite{MR2817509}]\label{Fact1}%
Given a partial acyclic edge coloring of $G$ and two colors $\alpha, \beta$, there exists at most one $(\alpha, \beta)$-maximal bichromatic  path containing a particular vertex $v$. \qed
\end{fact}

\begin{fact}[Basavaraju \etal \cite{MR2817509}]\label{Fact2}%
Let $G$ be a $\kappa$-deletion-minimal graph and $uv$ be an edge of $G$. If $\phi$ is an acyclic edge coloring of $G - uv$, then no available color for $uv$ is valid. Furthermore, if $\mathcal{U}(u) \cap \mathcal{U}(v) = \emptyset$, then $\deg(u) + \deg(v) = \kappa + 2$; if $|\mathcal{U}(u) \cap \mathcal{U}(v)| = s$, then $\deg(u) + \deg(v) + \sum\limits_{w \in W(u, v)} \deg(w) \geq \kappa+2s+2$. \qed
\end{fact}

We collect some structural lemmas on $\kappa$-deletion-minimal graphs, which are useful for our main result. 
\begin{lemma}\label{delta2}
If $G$ is a $\kappa$-deletion-minimal graph, then $G$ is $2$-connected and $\delta(G) \geq 2$.
\end{lemma}

The following two lemmas deal with the local structures of the $2$-vertices
\begin{lemma}[Wang and Zhang \cite{MR3166127}]\label{2+edge}%
Let $G$ be a $\kappa$-deletion-minimal graph. If $v$ is adjacent to a $2$-vertex $v_{0}$ and $N_{G}(v_{0}) = \{w, v\}$, then $v$ is adjacent to at least $\kappa - \deg(w) + 1$ vertices of degree at least $\kappa - \deg(v) + 2$. Moreover,
\begin{enumerate}[label=(\Alph*)]%
\item\label{A} if $\kappa \geq \deg(v) + 1$ and $wv \in E(G)$, then $v$ is adjacent to at least $\kappa - \deg(w) + 2$ vertices of degree at least $\kappa - \deg(v) + 2$, and $\deg(v) \geq \kappa - \deg(w) + 3$;
\item\label{B} if $\kappa \geq \Delta(G) + 2$ and $v$ is adjacent to precisely $\kappa - \Delta(G) + 1$ vertices of degree at least $\kappa - \Delta(G) + 2$, then $v$ is adjacent to at most $\deg(v) + \Delta(G) - \kappa - 3$ vertices of degree two and $\deg(v) \geq \kappa - \Delta(G) + 4$. \qed
\end{enumerate}
\end{lemma}

\begin{lemma}[Wang and Zhang \cite{MR3166127}]\label{2++edge}%
Let $G$ be a $\kappa$-deletion-minimal graph with $\kappa \geq \Delta(G) + 2$. If $v_{0}$ is a $2$-vertex, then $v_{0}$ is adjacent to two $(\kappa - \Delta(G) + 4)^{+}$-vertices.
\end{lemma}

Wang and Zhang also gave the following local structure of the $3$-vertices. 
\begin{lemma}[Wang and Zhang \cite{MR3166127}]\label{Good-3-vertex}%
Let $G$ be a $\kappa$-deletion-minimal graph with $\kappa \geq \Delta(G) + 2$ and $v$ be a $3$-vertex with $N_{G}(v) = \{w, v_{1}, v_{2}\}$. If $\deg(w) = \kappa - \Delta(G) + 2$, then $G$ has the following properties:

\begin{enumerate}[label= (\alph*)]%
\item\label{3a} there is exactly one common color at $w$ and $v$ for any acyclic edge coloring of $G - wv$. By symmetry, we may assume that the color on $vv_{1}$ is the common color;
\item\label{3b} $\deg(v_{1}) = \Delta(G) \geq \deg(v_{2}) \geq \kappa -\Delta(G) + 3$;
\item\label{3c} the edge $wv$ is not contained in any triangle in $G$ and $w$ is adjacent to exactly one $3^{-}$-vertex;
\item\label{3d} the vertex $v_{1}$ is adjacent to at least $\kappa - \deg(v_{2}) + 1$ vertices of degree at least $\kappa - \Delta(G) + 2$;
\item\label{3e} the vertex $v_{2}$ is adjacent to at least $\kappa - \Delta(G)$ vertices of degree at least $\kappa - \deg(v_{2}) + 2$;
\item\label{3f} the vertex $v_{2}$ is adjacent to at least $\kappa - \Delta(G) + 1$ vertices of degree at least four. \qed
\end{enumerate}
\end{lemma}

\begin{lemma}[Hou \etal \cite{MR2849391}]\label{3+vertex}%
If $G$ is a $\kappa$-deletion-minimal graph with $\kappa \geq \Delta(G) + 2$, then every $3$-vertex is adjacent to three $(\kappa - \Delta(G) + 2)^{+}$-vertices. 
\end{lemma}

\begin{lemma}\label{K_D_4}%
If $G$ is a $\kappa$-deletion-minimal graph with $\kappa \geq \Delta(G) + 2$, then every vertex is adjacent to at least two $4^{+}$-vertices. 
\end{lemma}

\begin{proof}
Let $w$ be a vertex with neighbors $w_{0}, w_{1}, \dots, w_{\tau - 1}$. Suppose to the contrary that $w$ is adjacent to at most one $4^{+}$-vertex. By \autoref{2+edge}, no $2$-vertex is adjacent to $w$. Let $w_{0}$ be a $3$-vertex with neighbors $w, v_{1}, v_{2}$. Since $G$ is $\kappa$-deletion-minimal, the graph $G - ww_{0}$ has an acyclic edge coloring $\phi$ with $\phi(ww_{i}) = i$ for $1 \leq i \leq \tau -1$.  Note that $\deg(w) + \deg(w_{0}) = \deg(w) + 3 \neq \kappa + 2$, Fact~2 guarantees $|\mathcal{U}(w) \cap \mathcal{U}(w_{0})| \geq 1$. Without loss of generality, we may assume that $w_{0}v_{1}$ is colored with $1$. 

\paragraph{\textsc{Case 1. $|\mathcal{U}(w) \cap \mathcal{U}(w_{0})| = 1$.}}
Note that $G$ cannot be acyclically edge colored with $\kappa$ colors, thus there exists a $(1, \alpha, w, w_{0})$-critical path for $\alpha \in A(ww_{0})$, and then $A(ww_{0}) \subseteq \Upsilon(w, w_{1}) \cap \Upsilon(w_{0}, v_{1})$. In this case, we consider the following two situations according to the degree of $w_{1}$. 

\paragraph{\textsc{Subcase 1.1. $w_{1}$ is a $3$-vertex.}} 
Recall that there exists a $(1, \alpha, w, w_{0})$-critical path for $\alpha \in A(ww_{0})$, thus $\tau = \kappa - 2 =\Delta$ and $\Upsilon(w, w_{1}) \subseteq \{\Delta, \Delta + 1, \Delta + 2\}$. If there exists another vertex $w_{s}$ with $\Upsilon(w, w_{s}) \subseteq \{\Delta, \Delta + 1, \Delta + 2\}$, then we can exchange the colors on $ww_{1}$ and $ww_{s}$, and additionally color $ww_{0}$ with an element in $A(ww_{0})$. Hence, we have $\Upsilon(w, w_{s}) \cap \{1, 2, \dots, \Delta - 1\} \neq \emptyset$ for $s \geq 2$. Let $w_{2}, w_{3}, \dots, w_{\tau - 2}$ be $3$-vertices. For $i \geq 2$, uncolor $ww_{i}$ and color $ww_{0}$ with $i$, we obtain an acyclic edge coloring $\phi_{i}$ of $G - ww_{i}$. 

There exists a $(\lambda_{\alpha}, \alpha, w, w_{2})$-critical path for each $\alpha$ in $A(ww_{2})$, for otherwise we can color $ww_{0}$ with $2$ and recolor $ww_{2}$ with $\alpha$. It follows that there exists $x, y \in A(ww_{2})$ with $\lambda_{x} = \lambda_{y} = \lambda$, and then $\{x, y\} \subseteq \Upsilon(w, w_{\lambda})$. By Fact~1, we have that $\lambda \neq 1$, and then the vertex $w_{\lambda}$ is a $4^{+}$-vertex and $\lambda = \tau - 1$. 

By similar arguments, we conclude that there exists a $(\tau - 1, \alpha_{1}, w, w_{3})$-critical path and a $(\tau - 1, \alpha_{2}, w, w_{3})$-critical path for some $\alpha_{1}, \alpha_{2}$ in $A(ww_{3})$. Note that $\{\alpha_{1}, \alpha_{2}\} \cup \{x, y\} \subseteq \{\Delta, \Delta + 1, \Delta + 2\}$, thus $\{\alpha_{1}, \alpha_{2}\} \cap \{x, y\} \neq \emptyset$, but this contradicts Fact~1. 

\paragraph{\textsc{Subcase 1.2. $w_{1}$ is a $4^{+}$-vertex.}}
Note that $|A(v_{1}) \cap \{2, 3, \dots, \tau - 1\}| \geq 1$, otherwise $\deg(v_{1}) \geq \kappa -1 \geq \Delta + 1$. By symmetry, we may assume that $2$ is a missing color at $v_{1}$. Uncolor $ww_{2}$ and color $ww_{0}$ with $2$, the resulting coloring is an acyclic edge coloring $\varphi$ of $G - ww_{2}$. By Fact~2, we have that $\Upsilon(w, w_{2}) \cap \{1, 2, \dots, \tau - 1\} \neq \emptyset$.  
\begin{itemize}
\item $\Upsilon(w, w_{2}) = \{\rho, \rho'\}$ and $\rho' \geq \tau$. There exists a $(\rho, \alpha, w, w_{2})$-critical path for $\alpha \in A(ww_{0}) \cap A(ww_{2})$, thus $\rho \neq 1$ and $w_{\rho}$ is a $3$-vertex. Now, we can reduce it to Subcase 1.1 with $\varphi$ playing the role of $\phi$. 

\item $\Upsilon(w, w_{2}) = \{\rho, \rho'\} \subseteq \{1, 2, \dots, \tau -1\}$. Note that none of $\tau, \tau + 1, \dots, \kappa$ is valid for $ww_{2}$ under $\varphi$, thus there exists a $(\rho, \alpha_{1}, w, w_{2})$-critical path and a $(\rho, \alpha_{2}, w, w_{2})$-critical path for some $\alpha_{1}, \alpha_{2}$ from $\{\tau, \tau + 1, \dots, \kappa\}$. It is obvious that $\{\alpha_{1}, \alpha_{2}\} \cap A(ww_{0}) \neq \emptyset$, thus $\rho \neq 1$ and $\Upsilon(w, w_{\rho}) = \{\alpha_{1}, \alpha_{2}\}$. So we may assume that $\rho = 3$. 

If $3 \notin \mathcal{U}(v_{1})$, then we can color $ww_{0}$ with $3$ and recolor $ww_{3}$ with an element in $\{\tau, \tau + 1, \dots, \kappa\} \setminus \{\alpha_{1}, \alpha_{2}\}$. It follows that $3 \in \mathcal{U}(v_{1})$. 

Suppose that $4 \notin \mathcal{U}(v_{1})$ and $\Upsilon(w, w_{4}) = \{p, q\}$.  For each $\alpha \in A(ww_{4})$, there exists a $(p, \alpha, w, w_{4})$-critical path or a $(q, \alpha, w, w_{4})$-critical path, for otherwise we can color $ww_{0}$ with $4$ and recolor $ww_{4}$ with $\alpha$. 
\begin{itemize} 
\item Suppose that $q \geq \tau$. Thus there exists a $(p, \alpha, w, w_{4})$-critical path for each $\alpha \in A(ww_{4}) $, and then $p \neq 3$ and $A(ww_{4}) \subseteq \Upsilon(w, w_{p})$. Note that $A(ww_{0}) \cap A(ww_{4}) \neq \emptyset$, so we have that $p \neq 1$. In fact $w_{p}$ is a $3$-vertex and $\Upsilon(w, w_{p}) = A(ww_{4})$. We can exchange the colors on $ww_{3}$ and $ww_{p}$, color $ww_{0}$ with $4$ and recolor $ww_{4}$ with an element in $A(ww_{4})$. 

\item Suppose that $\{p, q\} \subseteq \{1, 2, \dots, \tau -1\}$. Note that $A(ww_{4}) = \{\tau, \tau + 1, \dots, \kappa\}$ and $|A(ww_{4})| \geq 3$, so we may assume that there exists a $(p, \xi_{1}, w, w_{4})$-critical path and a $(p, \xi_{2}, w, w_{4})$-critical path for some $\xi_{1}, \xi_{2} \in A(ww_{4})$. It concludes that $\{\xi_{1}, \xi_{2}\} \subseteq \Upsilon(w, w_{p})$. Clearly, $\{\xi_{1}, \xi_{2}\} \cap A(ww_{0}) \neq \emptyset$,  and then $p \in \{1, 3\}$ due to Fact~1. So $w_{p}$ is a $3$-vertex with $\Upsilon(w, w_{p}) = \{\xi_{1}, \xi_{2}\}$. We can exchange the colors on $ww_{3}$ and $ww_{p}$, color $ww_{0}$ with $2$ and uncolor $ww_{2}$, and then we obtain a new acyclic edge coloring of $G - ww_{2}$. Under this new coloring, there exists a $(\rho', \alpha, w, w_{2})$-critical path for each $\alpha \in A(ww_{2})$, and then $A(ww_{2}) \subseteq \Upsilon(w, w_{\rho'})$. If $\rho' = 1$, then there exists a $(1, \alpha, w, w_{2})$-critical path and a $(1, \alpha, w, w_{0})$-critical path for each $\alpha \in A(ww_{0})$, which contradicts Fact~1. If $\rho' \neq 1$, then $\deg(w_{\rho'}) \geq 1 + |A(ww_{2})| \geq 4$, a contradiction. 
\end{itemize}

By similar arguments, $\{4, 5, \dots, \tau - 1\} \subseteq \mathcal{U}(v_{1})$, and then $\{1, 3, 4, \dots, \tau - 1\} \cup A(ww_{0}) \subseteq \mathcal{U}(v_{1})$. It follows that $A(v_{1}) = \{2, \phi(w_{0}v_{2})\}$. We recolor $w_{0}v_{1}$ with $2$, and then we reduce it to Subcase 1.1. 
\end{itemize}

\paragraph{\textsc{Case 2. $|\mathcal{U}(w) \cap \mathcal{U}(w_{0})| = 2$.}}
By symmetry, we may assume that $w_{0}v_{2}$ is colored with $2$. There exists a $(1, \alpha, w_{0}, w)$-critical path or a $(2, \alpha, w_{0}, w)$-critical path for $\alpha \in \{\tau, \tau + 1, \dots, \kappa\}$, thus $\{\tau, \tau + 1, \dots, \kappa\} \subseteq \Upsilon(w, w_{1}) \cup \Upsilon(w, w_{2})$.

\paragraph{\textsc{Subcase 2.1. Either $\Upsilon(w_{0}, v_{1}) \nsupseteq \{\tau, \dots, \kappa\}$ or $\Upsilon(w_{0}, v_{2}) \nsupseteq \{\tau, \dots, \kappa\}$.}}
By symmetry, we may assume that $\tau \notin \Upsilon(w_{0}, v_{2})$. Note that $\tau$ is not valid for $ww_{0}$, thus there exists a $(1, \tau, w_{0}, w)$-critical path, and then there exists no $(1, \tau, w_{0}, v_{2})$-critical path. Recoloring $w_{0}v_{2}$ with $\tau$ results in a new acyclic edge coloring $\sigma$ of $G - ww_{0}$ with $|\mathcal{U}_{\sigma}(w) \cap \mathcal{U}_{\sigma}(w_{0})| = 1$ and it takes us back to Case 1.

\paragraph{\textsc{Subcase 2.2. $\Upsilon(w_{0}, v_{1}) \supseteq \{\tau, \dots, \kappa\}$ and $\Upsilon(w_{0}, v_{2}) \supseteq \{\tau, \dots, \kappa\}$.}}

\paragraph{($\ast_{1}$)} Note that at most one of $w_{1}$ and $w_{2}$ is a $4^{+}$-vertex, so we may assume that $w_{2}$ is a $3$-vertex. If $\Upsilon(w, w_{2}) \subseteq \{\tau + 1, \tau + 2, \dots, \kappa\}$, then we can recolor $ww_{2}$ with an element in $\{\tau, \tau + 1, \dots, \kappa\} \setminus \Upsilon(w, w_{2})$, and then reduce it to Case~1. So we may assume that $\Upsilon(w, w_{2}) = \{\rho, \rho'\}$ with $\rho < \tau$. There exists a $(1, \alpha, w, w_{0})$-critical path for $\alpha \in \{\tau, \tau + 1, \dots, \kappa\} \setminus \{\rho'\}$, thus $\{\tau, \tau + 1, \dots, \kappa\} \setminus \{\rho'\} \subseteq \Upsilon(w, w_{1})$. If $w_{1}$ is a $3$-vertex, then $\Upsilon(w, w_{1}) \cup \{\rho'\} = \{\tau, \tau + 1, \dots, \kappa\} = \{\Delta, \Delta + 1, \Delta + 2\}$, and then we can recolor $ww_{1}$ with an element in $\{\tau, \tau + 1, \dots, \kappa\} \setminus \Upsilon(w, w_{1})$ and reduce it to Case~1. Hence, $w_{1}$ is a $4^{+}$-vertex. 
 
\paragraph{($\ast_{2}$)} Since $A(v_{1}) \cap \{3, 4, \dots, \tau - 1\} \neq \emptyset$, we may assume that $3$ is a missing color at $v_{1}$. Recoloring $w_{0}v_{1}$ with $3$ must create a $(3, 2)$-bichromatic cycle containing $w_{0}v_{1}$, for otherwise the resulting coloring is a new acyclic edge coloring of $G - ww_{0}$, and then one of $w_{2}$ and $w_{3}$ must be a $4^{+}$-vertex by a similar argument in the last paragraph. Let $\sigma$ be obtained by uncoloring $ww_{3}$ and coloring $ww_{0}$ with $3$. It is obvious that $\sigma$ is an acyclic edge of $G - ww_{3}$. We can conclude that $\Upsilon(w, w_{3}) \subseteq \{1, 2, \dots, \tau - 1\}$, otherwise we reduce it to Case~1. 

\paragraph{($\ast_{3}$)} Let $\Upsilon(w, w_{3}) = \{p, q\}$. By a similar argument as in the paragraph marked with ($\ast_{1}$), one of $w_{p}$ and $w_{q}$ is a $4^{+}$-vertex. So we may assume that $p = 1$. For $\alpha \in \{\tau, \tau + 1, \dots, \kappa\} \setminus \{\rho'\}$, there exists no $(1, \alpha, w, w_{3})$-critical path, so there exists a $(q, \alpha, w, w_{3})$-critical path and $\{\tau, \tau + 1, \dots, \kappa\} \setminus \{\rho'\} \subseteq \Upsilon(w, w_{q})$. Since $w_{q}$ is a $3$-vertex, thus $\Upsilon(w, w_{q}) \cup \{\rho'\} = \{\tau, \tau + 1, \dots, \kappa\} = \{\Delta, \Delta + 1, \Delta + 2\}$. Hence, there exists a $(1, \rho', w, w_{3})$-critical path, for otherwise we can color $ww_{0}$ with $3$ and recolor $ww_{3}$ with $\rho'$. 

Recall that recoloring $w_{0}v_{1}$ with $3$ creates a $(3, 2)$-bichromatic cycle containing $w_{0}v_{1}$, this implies that $2 \in \mathcal{U}(v_{1})$ and $|A(v_{1}) \cap \{3, 4, \dots, \tau -1\}| \geq 2$. So we may assume that $4$ is also a missing color at $v_{1}$. By a similar argument as in the paragraphs marked with ($\ast_{2}$) and ($\ast_{3}$), there exists a $(1, \rho', w, w_{4})$-critical path, but this contradicts Fact~1. 
\resetcounter
\end{proof}

In \cite{MR3453934}, Wang and Zhang presented the following structural lemma of the $4$-vertices. 
\begin{lemma}[Wang and Zhang \cite{MR3453934}]\label{4Sum}%
Let $G$ be a $\kappa$-deletion-minimal graph with $\kappa \geq \Delta(G) + 2$, and let $w_{0}$ be a $4$-vertex with $N_{G}(w_{0}) = \{w, v_{1}, v_{2}, v_{3}\}$.
\begin{enumerate}[label = (\Alph*)]%
\item\label{4Suma} If $\deg_{G}(w) \leq \kappa - \Delta(G)$, then
\begin{equation}%
\sum_{x \in N_{G}(w_{0})} \deg_{G}(x) \geq 2\kappa - \deg_{G}(w_{0}) + 8 = 2\kappa + 4.
\end{equation}
\item\label{4Sumb} If $\deg_{G}(w) \leq \kappa - \Delta(G) + 1$ and $ww_{0}$ is contained in two triangles, then
\begin{equation}\label{EQ2}%
\sum_{x \in N_{G}(w_{0})} \deg_{G}(x) \geq 2\kappa - \deg_{G}(w_{0}) + 9 = 2\kappa + 5.
\end{equation}
Furthermore, if the equality holds in \eqref{EQ2}, then all the other neighbors of $w$ are $6^{+}$-vertices. \qed
\end{enumerate}
\end{lemma}

\section{Proof of \autoref{M}}
Now, we are ready to prove the main result in this paper. 
\begin{proof}[Proof of \autoref{M}]
Suppose that $G$ is a counterexample to the theorem in the sense that $|V| + |E|$ is minimum. It is easy to see that $G$ is a $\kappa$-deletion-minimal graph, where $\kappa := \Delta(G) + 16$. By \autoref{delta2}, the graph $G$ is $2$-connected and $\delta(G) \geq 2$. 

Since $G$ is $\kappa$-deletion-minimal, it has the following local structures. 
\begin{enumerate}[label = (C\arabic*)]
\item\label{C1} Every $2$-vertex is adjacent to two $20^{+}$-vertices (\autoref{2++edge}). 
\item\label{C2} Every $3$-vertex is adjacent to three $18^{+}$-vertices (\autoref{3+vertex}). 
\item\label{C3} Every $18$-vertex is adjacent to at most one $3$-vertex (\autoref{Good-3-vertex}). 
\item\label{C4} Every vertex is adjacent to at least two $4^{+}$-vertices (\autoref{K_D_4}). 
\item\label{C5} Every $4$-vertex is adjacent to four $10^{+}$-vertices, or a $9^{-}$-vertex and three $22^{+}$-vertices (\autoref{4Sum}~\ref{4Suma}). 
\end{enumerate}

Suppose that $G$ contains a $5$-vertex $v$ adjacent to three $7^{-}$-vertices. Let $N_{G}(v) = \{u, v_{1}, v_{2}, v_{3}, v_{4}\}$ with $\deg(u) \leq 7, \deg(v_{1}) \leq 7$ and $\deg(v_{2}) \leq 7$. By the minimality of $G$, the graph $G - uv$ has an acyclic edge coloring $\phi$ with $\Delta(G) + 16$ colors. Moreover, when we choose the acyclic edge coloring $\phi$, we assume that the number of common colors on the edges incident with $u$ and $v$ is minimum, that is, $|\mathcal{U}(u) \cap \mathcal{U}(v)| = m$ is minimum among all the acyclic edge colorings of $G - uv$. We can easily obtain that $m \geq 1$ from Fact 2. Let $N_{G}(u) = \{v, u_{1}, \dots, u_{t}\}, t \leq 6$.

The first case: $m = 1$. Assume that $uu_{1}$ and $vv_{1}$ have the same color $1$. Note that all the available colors for $uv$ are invalid, hence there exists a $(1, \alpha, u, v)$-critical path for each $\alpha$ in $A(uv)$, and thus $A(uv) \subseteq \mathcal{U}(u_{1})$. But $|A(uv)| \geq \kappa - (6 + 4 - 1) > \Delta$, thus $|\mathcal{U}(u_{1})| \geq |A(uv)| + 1 > \Delta$, a contradiction. 

The second case: $m \geq 2$. Assume that $uu_{i}$ and $vv_{i}$ have the same color $i$ for each $i \in \{1, 2, \dots, m\}$. For each $\alpha \in A(uv)$, there exists an $(i_{\alpha}, \alpha, u, v)$-critical path for some $i_{\alpha} \in \{1, 2, \dots, m\}$. Note that $|A(uv)| \geq \kappa - 10 + m \geq \kappa - 8 \geq \Delta + 8$, 
\[
\deg(v_{1}) + \deg(v_{2}) + \deg(v_{3}) + \deg(v_{4}) - 4 < 2 |A(uv)|,
\]
thus there exists an available color $\alpha^{*}$ such that it appears exactly once in $\mathbb{S}$, where $\mathbb{S}$ is defined as $\mathbb{S} : = \Upsilon(v, v_{1}) \uplus \Upsilon(v, v_{2}) \uplus \Upsilon(v, v_{3}) \uplus \Upsilon(v, v_{4})$. Without loss of generality, we may assume that it appears in $\mathcal{U}(v_{1})$, and then there exists a $(1, \alpha^{*}, u, v)$-critical path. Now, we revise $\phi$ by recoloring $vv_{2}$ with $\alpha^{*}$, which yields a new acyclic edge coloring of $G - uv$, but it contradicts the minimality of $m$. Therefore, the graph $G$ does not contain a $5$-vertex adjacent to three $7^{-}$-vertices. That is,  

\begin{enumerate}
\item[(C6)]\label{C6} every $5$-vertex is adjacent to at least three $8^{+}$-vertices. 
\end{enumerate}

{\bf Discharging Part.} In the following, we may assume that $G$ has been embedded on the plane such that every edge is crossed by at most one other edge. Moreover, the number of crossings is as small as possible. We treat each of the crossings as a vertex and obtain an {\em associated plane graph} $G^{\dagger}$. 

Since the number of crossings is as small as possible in the embedding, we have the following claim.
\begin{claim}
Every $2$-vertex is incident with two $4^{+}$-faces in $G^{\dagger}$. 
\end{claim}

Since $G$ is triangle-free and every $2$-vertex is incident with two $4^{+}$-faces in $G^{\dagger}$, we have the following statement. A similar statement has been proven in \cite{MR3397083}. 
\begin{claim}
Every $\ell$-vertex is incident with at most $\Lfloor \frac{2(\ell - \lambda)}{3} \Rfloor$ $3$-faces in $G^{\dagger}$, where $\lambda$ is the number of adjacent $2$-vertices. 
\end{claim}

We use the discharging method to derive a contradiction. Here, we need the following rewritten Euler's formula for the associated plane graph $G^{\dagger}$: 
\begin{equation}%
\sum_{v \in V(G^{\dagger})}(\deg(v) - 4) + \sum_{f \in F(G^{\dagger})}(\deg(f) - 4) = - 8.
\end{equation}

At first, we assign the initial charge of every vertex $v$ to be $\deg(v) - 4$ and the initial charge of every face $f$ to be $\deg(f) - 4$. Next, we design appropriate discharging rules and redistribute charges among vertices and faces, such that the final charge of every vertex and every face is nonnegative, which leads to a contradiction. Note that all the adjacencies between vertices in the discharging rules are refer to the adjacencies between vertices in $G$, not in $G^{\dagger}$.

{\bf Discharging rules:}
\begin{enumerate}[label = (R\arabic*)]
\item every $2$-vertex receives $1$ from each adjacent vertex; 
\item every $3$-vertex receives $\frac{3}{2}$ from the adjacent $18$-vertex;
\item every $3$-vertex receives $\frac{1}{2}$ from each adjacent $19^{+}$-vertex; 
\item every $3$-vertex receives $\frac{1}{2}$ from each incident $5^{+}$-face; 
\item every $3$-face receives $\frac{1}{2}$ from each incident non-crossing vertex; 
\item every non-crossing $4$-vertex receives $\frac{1}{4}$ from each adjacent vertex if it is adjacent to four $10^{+}$-vertices;
\item every non-crossing $4$-vertex receives $\frac{1}{3}$ from each adjacent $22^{+}$-vertex if it is adjacent to a $9^{-}$-vertex and three $22^{+}$-vertices;
\item every $5$-vertex receives $\frac{1}{6}$ from each adjacent $8^{+}$-vertex.
\end{enumerate}

If $w$ is an arbitrary $2$-vertex, then its final charge is $2 - 4 + 2 \times 1 = 0$. Let $w$ be an arbitrary $3$-vertex. Note that $w$ is incident with at most two $3$-faces. If $w$ is incident with at most one $3$-face, then its final charge is at least $3 - 4 + \frac{3}{2} - \frac{1}{2} = 0$. On the other hand, if $w$ is incident with exactly two $3$-faces, then it is incident with a $5^{+}$-face and it receives $\frac{1}{2}$ from the $5^{+}$-face, and then its final charge is at least $3 - 4 + \frac{3}{2} + \frac{1}{2} - 2 \times \frac{1}{2} = 0$. Hence, the final charge of an arbitrary $3$-vertex is nonnegative. 

It is obvious that the final charge of a crossing $4$-vertex is zero. Let $w$ be an arbitrary non-crossing $4$-vertex. If $w$ is adjacent to four $10^{+}$-vertices, then its final charge is at least $4 - 4 + 4 \times \frac{1}{4} - 2 \times \frac{1}{2} = 0$. If $w$ is adjacent to a $9^{-}$-vertex and three $22^{+}$-vertices, then its final charge is at least $4 - 4 + 3 \times \frac{1}{3} - 2 \times \frac{1}{2} = 0$. 

If $w$ is a $5$-vertex, then it is adjacent to at least three $8^{+}$-vertices, then its final charge is at least $5 - 4 + 3 \times \frac{1}{6} - 3 \times \frac{1}{2} = 0$. 

If $w$ is an arbitrary $\ell$-vertex with $\ell = 6, 7$, then its final charge is at least $\ell - 4 - \frac{2\ell}{3} \times \frac{1}{2} \geq 0$. 

If $w$ is an arbitrary $\ell$-vertex with $\ell = 8, 9$, then its final charge is at least $\ell - 4 - \frac{2\ell}{3} \times \frac{1}{2} - \frac{1}{6} \ell \geq 0$.

Let $w$ be an arbitrary $10^{+}$-vertex in the following. Suppose that $w$ is adjacent to at least one $2$-vertex. Let $\lambda$ be the number of adjacent $2$-vertices. By \autoref{2+edge}, it is adjacent to at least seventeen $18^{+}$-vertices, thus its final charge is at least $\ell - 4 - \frac{2(\ell - \lambda)}{3}  \times \frac{1}{2} - \lambda \times 1 - (\ell - \lambda -17) \times \frac{1}{2} = \frac{\ell - \lambda}{6} + \frac{9}{2} > 0$. So we may assume that $w$ is not adjacent to any $2$-vertex. 

\begin{itemize}
\item If $w$ is an $\ell$-vertex with $\ell \geq 22$, then its final charge is at least $\ell - 4 - \left\lfloor \frac{2\ell}{3} \right\rfloor \times \frac{1}{2} - \ell \times \frac{1}{2} \geq 0$. 

\item If $w$ is an $\ell$-vertex with $\ell = 19, 20, 21$, then its final charge is least $19 - 4 - \left\lfloor \frac{2 \times 19}{3} \right\rfloor \times \frac{1}{2} - 17 \times \frac{1}{2} - 2 \times \frac{1}{4} = 0$, $20 - 4 - \left\lfloor \frac{2 \times 20}{3} \right\rfloor \times \frac{1}{2} - 18 \times \frac{1}{2} - 2 \times \frac{1}{4} = 0$, or $21 - 4 - \left\lfloor \frac{2 \times 21}{3} \right\rfloor \times \frac{1}{2} - 19 \times \frac{1}{2} - 2 \times \frac{1}{4} = 0$. 

\item If $w$ is an $18$-vertex, then it is adjacent to at most one $3$-vertex, and then its final charge is at least $18 - 4 - \frac{2 \times 18}{3} \times \frac{1}{2} - \frac{3}{2} - 17 \times \frac{1}{4} > 0$. 

\item If $w$ is an $\ell$-vertex with $10 \leq \ell \leq 17$, then it is only adjacent to $4^{+}$-vertices, and then its final charge is at least $\ell - 4 - \frac{2\ell}{3} \times \frac{1}{2} - \ell \times \frac{1}{4} > 0$.
\end{itemize}

It is obvious that every $3$-face has the final charge $3 - 4 + 2 \times \frac{1}{2} = 0$. Every $4$-face has the final charge as its initial charge, zero. By \ref{C2}, there is no consecutive $3$-vertices lying on a face boundary, thus every $5^{+}$-face $f$ has final charge at least $\deg(f) - 4 - \Lfloor \frac{\deg(f)}{2} \Rfloor \times \frac{1}{2} \geq 0$.

Now, the final charge of every vertex and every face is nonnegative, which derives the desired contraction. 
\resetcounter
\end{proof}
\vskip 0mm \vspace{0.3cm} \noindent{\bf Acknowledgments.} This project was supported by the National Natural Science Foundation of China (11101125) and partially supported by the Fundamental Research Funds for Universities in Henan.


\begin{thebibliography}{10}

\bibitem{MR1109695}
N.~Alon, C.~McDiarmid and B.~Reed, Acyclic coloring of graphs, Random
  Structures Algorithms 2~(3) (1991) 277--288.

\bibitem{MR1837021}
N.~Alon, B.~Sudakov and A.~Zaks, Acyclic edge colorings of graphs, J. Graph
  Theory 37~(3) (2001) 157--167.

\bibitem{MR2817509}
M.~Basavaraju, L.~S. Chandran, N.~Cohen, F.~Havet and T.~M{\"u}ller, Acyclic
  edge-coloring of planar graphs, SIAM J. Discrete Math. 25~(2) (2011)
  463--478.

\bibitem{MR3037985}
L.~Esperet and A.~Parreau, Acyclic edge-coloring using entropy compression,
  European J. Combin. 34~(6) (2013) 1019--1027.

\bibitem{MR526851}
I.~Fiam{\v{c}}{\'{\i}}k, The acyclic chromatic class of a graph, Math. Slovaca
  28~(2) (1978) 139--145.

\bibitem{MR2458434}
A.~Fiedorowicz, M.~Ha{\l}uszczak and N.~Narayanan, About acyclic edge
  colourings of planar graphs, Inform. Process. Lett. 108~(6) (2008) 412--417.

\bibitem{MR3448625}
I.~Giotis, L.~Kirousis, K.~I. Psaromiligkos and D.~M. Thilikos, On the
  algorithmic {L}ov\'asz local lemma and acyclic edge coloring, in: 2015
  {P}roceedings of the {T}welfth {W}orkshop on {A}nalytic {A}lgorithmics and
  {C}ombinatorics ({ANALCO}), SIAM, Philadelphia, PA, 2015, pp. 16--25.

\bibitem{MR3031510}
Y.~Guan, J.~Hou and Y.~Yang, An improved bound on acyclic chromatic index of
  planar graphs, Discrete Math. 313~(10) (2013) 1098--1103.

\bibitem{MR2849391}
J.~Hou, N.~Roussel and J.~Wu, Acyclic chromatic index of planar graphs with
  triangles, Inform. Process. Lett. 111~(17) (2011) 836--840.

\bibitem{MR2491777}
J.~Hou, J.~Wu, G.~Liu and B.~Liu, Acyclic edge colorings of planar graphs and
  series-parallel graphs, Sci. China Ser. A 52~(3) (2009) 605--616.

\bibitem{MR1715600}
M.~Molloy and B.~Reed, Further algorithmic aspects of the local lemma, in:
  Proceedings of the Thirtieth Annual ACM Symposium on the Theory of Computing,
  ACM, New York, 1998, pp. 524--529.

\bibitem{MR2864444}
S.~Ndreca, A.~Procacci and B.~Scoppola, Improved bounds on coloring of graphs,
  European J. Combin. 33~(4) (2012) 592--609.

\bibitem{MR0187232}
G.~Ringel, {Ein Sechsfarbenproblem auf der Kugel}, Abh. Math. Sem. Univ.
  Hamburg 29~(1) (1965) 107--117.

\bibitem{MR3065111}
Q.~Shu, W.~Wang and Y.~Wang, Acyclic chromatic indices of planar graphs with
  girth at least 4, J. Graph Theory 73~(4) (2013) 386--399.

\bibitem{MR3397083}
W.~Song and L.~Miao, Acyclic edge coloring of triangle-free 1-planar graphs,
  Acta Math. Sin. (Engl. Ser.) 31~(10) (2015) 1563--1570.

\bibitem{MR3166127}
T.~Wang and Y.~Zhang, Acyclic edge coloring of graphs, Discrete Appl. Math. 167
  (2014) 290--303.

\bibitem{MR3453934}
T.~Wang and Y.~Zhang, Further result on acyclic chromatic index of planar
  graphs, Discrete Appl. Math. 201 (2016) 228--247.

\bibitem{MR2994403}
W.~Wang, Q.~Shu and Y.~Wang, A new upper bound on the acyclic chromatic indices
  of planar graphs, European J. Combin. 34~(2) (2013) 338--354.

\end{thebibliography}
\end{document}